\documentclass[oneside,a4paper,11pt,notitlepage]{article}
\usepackage[left=0.8in, right=0.8in, top=1.5in, bottom=1.5in]{geometry}

\usepackage[T1]{fontenc} % codifica dei font
\usepackage[utf8]{inputenc} % lettere accentate da tastiera
\usepackage[english]{babel}
\usepackage{amssymb}
\usepackage{lipsum} % genera testo fittizio
\usepackage{lmodern}
\usepackage{amsmath}
\usepackage{amsthm}
\usepackage{bm}
\usepackage{mathtools}
\usepackage{braket}
\usepackage{esint}
\newcommand{\abs}[1]{{\left|#1\right|}}
\newcommand{\norma}[1]{{\left\Vert#1\right\Vert}}
\newcommand{\normab}[1]{{\big\lVert#1\big\rVert}}

\usepackage{enumitem}
\usepackage{booktabs}
\usepackage{graphicx}
\usepackage{tikz}
\usetikzlibrary{patterns}
\usepackage{multicol}
\usepackage{caption}
\usepackage[skins,theorems]{tcolorbox}
\tcbset{highlight math style={enhanced,
colframe=black,colback=white,arc=0pt,boxrule=1pt}}
\def\XXint#1#2#3{{\setbox0=\hbox{$#1{#2#3}{\int}$}
    \vcenter{\hbox{$#2#3$}}\kern-.5\wd0}}

\theoremstyle{definition}
\newtheorem{definizione}{Definition}[section]
\theoremstyle{plain}
\newtheorem{theorem}{Theorem}[section]

\newtheorem{lemma}[theorem]{Lemma}
\newtheorem{prop}[theorem]{Proposition}
\newtheorem{corollario}[theorem]{Corollary}
\theoremstyle{definition}
\newtheorem{esempio}{Example}[section]
\newtheorem{oss}[esempio]{Remark}

\newtheorem*{open*}{Open problems}
\DeclareMathOperator{\R}{\mathbb{R}}

\makeatletter
\newcommand{\myfootnote}[2]{\begingroup
	\def\@makefnmark{}%
	\addtocounter{footnote}{-1}%
	\footnote{\textbf{#1} #2}
	\endgroup}
\makeatother

\usepackage{hyperref}
\hypersetup{linktoc=none, bookmarksnumbered, colorlinks=true, linkcolor=magenta, citecolor=cyan}

\title{Quantitative comparison results for first-order Hamilton-Jacobi equations}
\author{Vincenzo Amato, Luca Barbato}
\date{}

\newcommand{\Addresses}{{% additional braces for segregating \footnotesize 
\bigskip 
  \footnotesize 
 \textit{E-mail address}, V. ~Amato (corresponding author): \texttt{v.amato@ssmeridionale.it}

     \medskip 

  \textit{E-mail address}, L.~Barbato: \texttt{l.barbato@ssmeridionale.it} 
  
   \medskip 
 
  \textsc{Mathematical and Physical Sciences for Advanced Materials and Technologies, Scuola Superiore Meridionale, Largo San Marcellino 10, 80138 Napoli, Italy. }

 \par\nopagebreak 

}} 

\begin{document}

\maketitle
 \begin{abstract}
%\color{violet}
In this paper, we study a quantitative refinement of a classical symmetrisation result for first-order Hamilton–Jacobi equations. 
We prove that the deficit in the comparison result, established by Giarrusso and Nunziante, controls both the asymmetry of the domain and the deviation of the solution and data from radial symmetry. 
This yields a stability version of the Giarrusso–Nunziante inequality.
\end{abstract}

\section{Introduction}
%\color{violet}
In this work, we address a classical question in shape optimisation: how the shape of a domain influences the solutions to certain partial differential equations. We focus on first-order Hamilton-Jacobi equations, a class of PDEs that arise in a wide range of applications, including control theory, differential geometry, and classical mechanics.

Our main tool is symmetrisation: the process of comparing a given function (or domain) with a more symmetric one, typically a ball, and analysing how this affects the solutions of the equation.

More precisely, given a bounded domain \( \Omega \subset \mathbb{R}^n \) and a non-negative function \( f \in L^p(\Omega) \) defined on it, we consider the equation
\begin{equation}
\label{HJ_u}
    |\nabla u| = f(x) \quad \text{in } \Omega, \qquad u = 0 \text{ on } \partial \Omega.
\end{equation}
We then study its symmetrised counterpart, where the domain \( \Omega \) is replaced by a ball \( \Omega^\sharp \) (with the same volume as \( \Omega \)), and the function \( f \) is replaced by its increasing rearrangement \( f_\sharp \) (see Section \ref{section2} for its definition). The unique radially symmetric and decreasing solution to this symmetrised problem is denoted by \( u_G \).

A classical result by Giarrusso and Nunziante, \cite{GN}, shows that the \( L^1 \)-norm of the original solution \( u \) is always less than or equal to that of the symmetric solution \( u_G \):

\begin{equation}
\label{GN}
    \big\lVert u\big\rVert_{L^1(\Omega)} \leq \big\lVert u^\text{G}\big\rVert_{L^1(\Omega^{\sharp})}.
\end{equation}

For completeness, we recall that a (\textbf{generalized}) solution
 $u$ to \eqref{HJ_u} is any function $u \in W^{1,p}_0(\Omega)$, $p \geq 1$, satisfying $\abs{\nabla u}= f(x)$ almost everywhere in $\Omega$. 
 
However, uniqueness is not guaranteed: for instance, in the case \( \Omega = [0, 1] \), and \( f(x) = 1 \), infinitely many solutions exist. To ensure uniqueness, one may refer to the theory of viscosity solutions (see~\cite{tran}); however, for the purposes of~\cite{GN} and of this paper, the existence of at least one such 
$u$ is sufficient, as every solution satisfies~\eqref{GN}.

An extension of this result to more general equations, involving terms such as \( f(x) - \lambda u \), was given in \cite{GN2}.

A natural question is whether other norms (besides the \( L^1 \)-norm) of the solutions to \eqref{HJ_u} and the symmetrised version can also be compared. In \cite{ALT,FP}, the authors provide a rearrangement of the gradient of \( u \) that allows such comparisons in the \( L^q \)-norm. A further characterisation of these rearrangements appears in \cite{C}, where it is shown that they coincide with the increasing rearrangement for \( q = 1 \) and the decreasing one for \( q = \infty \).

In \cite{BM, M}, the rigidity of inequality \eqref{GN} is examined. In particular, the authors show that if equality holds in \eqref{GN}, then \( \Omega \), up to translation, must be a ball, and both \( u \) and \( f \) must be radially symmetric.

This naturally raises the question: can inequality \eqref{GN} be quantified? More precisely, if the deficit
\begin{equation}
    \label{deficit}
    \normab{u^\text{G}}_{L^1(\Omega^{\sharp})}-\normab{u}_{L^1(\Omega)} 
\end{equation}

is small, does this imply that \( \Omega \) is close to a ball, and that \( u \) and \( f \) are approximately radially symmetric?

To measure how far a domain \( \Omega \) is from being a ball, we use the Fraenkel asymmetry (see \cite{F}):

\begin{equation*}
	% \label{asimm}
    \alpha(\Omega):=\min_{x \in \R^{n}}\bigg \{  \dfrac{|\Omega\triangle B_r(x)|}{|B_r(x)|} \;,\; |B_r(x)|=|\Omega|\bigg \},
\end{equation*}
where \( \triangle \) denotes the symmetric difference.

Measuring how far a function is from being radial is more subtle. In related works on quantitative versions of the Pólya–Szegő, Hardy–Littlewood, and Talenti inequalities (see \cite{ABMP,CFET,CF}), the distance of a function \( g \) from radial symmetry is often measured by the \( L^p \)-norm of the difference between \( g \) and its decreasing rearrangement:

\[
\lVert g - g^\sharp\rVert_{L^p(\mathbb{R}^n)},
\]

where \( g \) is extended by zero outside \( \Omega \). We adopt this convention as well. 

The two Hamilton–Jacobi equations we consider are as follows:

\begin{subequations}
    \noindent\begin{minipage}{0.48\textwidth}
\begin{equation}
\begin{cases}
                    \abs{\nabla u} = f(x) & \text{a.e. in } \Omega \\
                    u = 0 & \text{on } \partial \Omega
                \end{cases}
                \label{original_u}
\end{equation}
    \end{minipage}
    \begin{minipage}{0.52\textwidth}
\begin{equation}
 \label{original_v}
 \quad
                \begin{cases}
            \big\lvert\nabla u^\text{G}\big\rvert = f_{\sharp}(x) & \text{a.e. in } \Omega^{\sharp} \\
                    u^\text{G} = 0 & \text{on } \partial \Omega^{\sharp}.
                \end{cases}
\end{equation}
    \end{minipage}\vskip1em
\end{subequations}
Our main theorem provides a quantitative version of inequality \eqref{GN}, showing that the deficit in the \( L^1 \) norm controls both the asymmetry of \( \Omega \) and the difference between \( u \) and \( u^\sharp \). More precisely:
\begin{theorem}
    \label{main_theorem} 
Let $\Omega$ be a bounded open set of $\R^n$, $n\geq 2$, and let $f\in L^\infty(\Omega)$ be a non-negative function with $$0<m\leq f\leq M.$$  
Let $u$ be a solution to \eqref{original_u} and let $u^\text{G}$ be the unique spherically decreasing solution to \eqref{original_v}.
Then there exist positive constants $ \theta=\theta(n)$ and    $ C_1:= C_1(n, \abs{\Omega}, m,M),\, C_2:=  C_2(n, \abs{\Omega}, m,M)$ such that
\begin{equation}\label{esti_tot}
     C_1 \alpha^3(\Omega)+ C_2 \inf_{x_0\in\R^n}\normab{u\pm u^\sharp(\cdot+x_0)}_{L^1(\R^n)}^{ \theta}\leq \normab{u^G}_{L^1(\Omega^{\sharp})}-\normab{u}_{L^1(\Omega)} .
\end{equation}
Moreover, the dependence of $ C_1$ on $n$, $\abs{\Omega}$, $m$ and $M$ is explicit.
\end{theorem}

This result is inspired by \cite[Theorem 1.1]{ABMP}, where the authors prove a quantitative version of the pointwise inequality by Talenti for solutions to the Poisson equation with Dirichlet boundary conditions. As in that result, our inequality shows that the deficit \eqref{deficit} controls both the distance of $\Omega$ from a ball and the distance of $u$ from its symmetrisation, with suitable powers. Whether these powers are optimal remains open. 

The assumptions on \( f \) may appear strong, but they are necessary for the propagation of the asymmetry introduced in \cite{BD,HN}, which is based on controlling the asymmetry of the level sets of \( u \). The variation of asymmetry across these level sets depends on the gradient of \( u \), and hence on \( f \) itself.

Regarding the almost radiality of \( f \), we obtain partial results. We show that the deficit controls the distance between \( f \) and a special rearrangement \( f_u \), which shares the same level sets as \( u \). (See Section \ref{section2} for the precise definition.)

\begin{prop}
    \label{main_prop} 
Let $\Omega$ be a bounded open set of $\R^n$, $n\geq 2$, and let $f\in L^\infty(\Omega)$ be a non-negative function with $$0<m\leq f\leq M.$$  
Let $u$ be solution to \eqref{original_u} and let $u^\text{G}$ be the unique spherically decreasing solution to \eqref{original_v}. Then there exists a positive constant    $C_3:=  C_3(n, \abs{\Omega}, M),$ such that
\begin{equation}\label{f-fu}
    C_3 \normab{f-f_u}_{L^1(\R^n)}^2\leq \normab{u^G}_{L^1(\Omega^{\sharp})}-\normab{u}_{L^1(\Omega)} .
\end{equation}
Moreover, the dependence of $C_3$ on $n$, $\abs{\Omega}$ and $M$ is explicit.
\end{prop}

As an intermediate step in proving the almost radiality of $u$ and $f$, we study the behaviour of the Schwarz symmetrisazion with respect to the rearrangement of the gradient. To be more precise we state the following theorem.

\begin{theorem}
\label{quantitbella}
      Let $\Omega$ be a bounded open set of $\R^n$, $n\geq 2$,  and let $u$ be a solution to \eqref{original_u} and $u^\text{G}$ be the unique spherically decreasing solution to \eqref{original_v}. Then 
        \begin{equation}
        \label{sharpG}
            \normab{u}_{L^1(\Omega)}\leq \normab{(u^\sharp)^\text{G}}_{L^1(\Omega^\sharp)} \leq \normab{u^\text{G}}_{L^1(\Omega^\sharp)}.
        \end{equation}
\end{theorem}
This result can be interpreted as a further refinement of inequality \eqref{GN},
since we can bound $\normab{u^\text{G}}_{L^1(\Omega)}-\normab{u}_{L^1(\Omega)}$ from below by a positive quantity.

\vspace{1mm}

For the sake of completeness, we recall that this result was later extended to more general equations involving additional terms of the form \( f(x) - \lambda u \) (see \cite{GN2}). Beyond the previously mentioned works, further related problems have been investigated in the literature. In particular, in \cite{AG,AGNT}, the results by Giarrusso and Nunziante were generalized to include functions with non-zero trace, as well as functions in the space of bounded variation (\(BV\)). Another important development was the analysis of the non-stationary case carried out in \cite{FPV}, which is of physical relevance since this is the form in which Hamilton–Jacobi equations typically appear in the literature or arise in applications, such as in Hamiltonian mechanics. Moreover, Talenti himself established analogous estimates in \cite{Ta6}, expressed in terms of the Lorentz norm \(L^{q,1}\).

The paper is organised as follows: in Section \ref{section2} we introduce some notation and preliminary results; Sections \ref{section3}, \ref{section4} and \ref{section5} are dedicated to establishing the boundedness of each term on the right-hand side of \eqref{esti_tot} and \eqref{f-fu}.

\color{black}

\section{Notation and preliminaries}
\label{section2}

Throughout the paper, \( |\Omega| \) denotes the Lebesgue measure of a measurable set \( \Omega \subset \mathbb{R}^n \), and \( P(E, \Omega) \) denotes the perimeter of a measurable set \( E \subset \mathbb{R}^n \) in \( \Omega \), defined by
\[
P(E, \Omega) := \sup \left\{ \int_E \mathrm{div}\, \varphi \, dx : \varphi \in C_c^\infty(\Omega),\ \|\varphi\|_\infty \le 1 \right\}.
\]
We write \( P(E) := P(E, \mathbb{R}^n) \) for the total perimeter of \( E \).

For every function $f\in W^{1,1}_{\text{loc}}(\Omega)$ and  $u:\Omega\to\R$ be a measurable function, the \textbf{coarea formula} holds, see for instance \cite{AFP,Fleming_Rishel,maggi2012sets}.

\begin{theorem}[Coarea formula]
	Let $\Omega \subset \mathbb{R}^n$ be an open set. Let $f\in W^{1,1}_{\text{loc}}(\Omega)$ and let $u:\Omega\to\R$ be a measurable function. Then,
	\begin{equation}
        	\label{coarea}
		{\displaystyle \int _{\Omega}u(x)|\nabla f(x)|dx=\int _{\mathbb {R} }dt\int_{\Omega\cap f^{-1}(t)}u(y)\, d\mathcal {H}^{n-1}(y)}.
	\end{equation}

Moreover, if  $f \in \text{BV}(\Omega)$, then the Fleming-Rishel formula holds, i.e.
\begin{equation}
	\label{flemingrishel}
	\abs{Df}(\Omega) = \int_{-\infty}^{+\infty} P(\Set{ f>t},\Omega ) \, dt.
\end{equation}
\end{theorem}

In the following, we discuss rearrangements of functions, emphasizing both their qualitative and quantitative aspects. For convenience in the proofs, we also outline several rescaling properties.

\subsection{Distribution function and rearrangements}
In this subsection, we introduce several definitions and properties related to function rearrangements. For a general overview of the basic results in this area, see, for instance, \cite{K}.
\begin{definizione}
    Let \( g \) be a non-negative measurable function on \( \Omega \subset \mathbb{R}^n \). The \textbf{distribution function} of \( g \) is defined as
\[
\mu_g(t) := |\{ x \in \Omega : g(x) > t \}|, \qquad t \ge 0.
\]
\end{definizione}

\begin{definizione} \label{decreasing:rear}
	Let $g: \Omega \to \R$ be a measurable function, the \textbf{decreasing rearrangement} of $g$, denoted by $g^\ast(\cdot)$, is defined as
$$g^*(s)=\inf\{t\geq 0:\mu_g(t)<s\}.$$
 and the \textbf{increasing rearrangement} of $g$, denoted by $g_\ast$, is defined as
    \begin{equation*}
        g_* \colon [0,\abs{\Omega}] \to \R^+ \qquad g_*(s) = g^*(\abs{\Omega}-s)
    \end{equation*}
	\end{definizione}

\begin{definizione}
The \textbf{spherically decreasing rearrangement} (or \textbf{Schwarz symmetrisation}) of \( g \) is the function \( g^\sharp : \mathbb{R}^n \to \mathbb{R} \) whose superlevel sets are centred balls having the same measure as those of \( g \). Equivalently,
\[
g^\sharp(x) := g^*(\omega_n |x|^n),
\]
where \( \omega_n \) denotes the measure of the unit ball in \( \mathbb{R}^n \).

The \textbf{spherically increasing rearrangement} of \( g \) is the function \( g_\sharp : \mathbb{R}^n \to \mathbb{R} \) whose superlevel sets are centred annuli having the same measure as those of \( g \). Equivalently,
\[
g_\sharp(x) := g_*(\omega_n |x|^n).
\]
\end{definizione}

These rearrangements are all \emph{equimeasurable}, i.e.\ they preserve the distribution function of \( g \), hence, by the Cavalieri principle, it holds
$$ \displaystyle{\norma{g}_{L^p(\Omega)}=\norma{g^*}_{L^p[0,\abs{\Omega}]}=\lVert{g^\sharp}\rVert_{L^p(\Omega^\sharp)}=\norma{g_*}_{L^p[0,\abs{\Omega}]}=\lVert{g_\sharp}\rVert_{L^p(\Omega^\sharp)}}, \quad \text{for all } p\ge1.$$

We summarize below some classical results on function rearrangements: the condition ensuring the absolute continuity of $\mu_g$ and two fundamental inequalities (see, for instance, \cite{BZ,CF2,HLP,K,PS}).

\begin{lemma}
    \label{asscontmu}
    Let $g\in W^{1,p}(\R^n)$, with $p\in(1,+\infty)$. 
    The following results hold:
    \begin{enumerate}
        \item \textbf{(Absolute continuity of $\mu_g$)}  
        the distribution function $\mu_g$ of $g$ is absolutely continuous if and only if 
\begin{equation*}
    \abs{\{0<g^\sharp<  ||g||_\infty\}\cap \{|\nabla g^\sharp| =0\}}=0.
\end{equation*}

        \item \textbf{(Hardy--Littlewood inequality)}  
        if $h \in L^{p'}(\Omega)$, then 
        \begin{equation}\label{hardy_littlewood}
 \int_{\Omega} \abs{h(x)g(x)} \, dx \le \int_{0}^{\abs{\Omega}} h^*(s) g^*(s) \, ds= \int_{0}^{\abs{\Omega}} h_*(s) g_*(s) \, ds.
\end{equation}

        \item \textbf{(P\'olya--Szegő inequality)}  
   if $g \in W^{1,p}_0(\Omega)$, then $g^{\sharp} \in W^{1,p}_0(\Omega^\sharp)$ and
	\begin{equation*}
		\lVert \nabla g^{\sharp} \rVert_{p} \leq \norma{\nabla g}_{p}.
	\end{equation*}The inequality also holds for $p = \infty$, that is,
        \begin{equation*}
            \|\nabla g^{\sharp}\|_{\infty} \le \|\nabla g\|_{\infty}.
        \end{equation*}
    \end{enumerate}
\end{lemma}

\subsection{Pseudo-rearrangements}

We now introduce the notion of pseudo-rearrangements, which generalizes classical rearrangements and was first studied in \cite{AT}.  
These constructions are useful when one wants to compare integrals over level sets of different functions.
\begin{definizione}[Pseudo-rearrangement]\label{pseudo}
    Let $g$ be a function in $L^p(\Omega)$, $\forall s  \in [0, |\Omega|]$ we can find a subset $D(s)$ such that:
    \begin{enumerate}
        \item $|D(s)|=s$;
        \item  $s_1 \leq s_2
        \implies D(s_1)\subset D(s_2)$;
        \item $D(s)= \{x \in \Omega : \, |g(x)| >t\}$ if $s = \mu_g(t)$.
    \end{enumerate}
    Then, for any $f\in L^p(\Omega)$, the function \[\int_{D(s)}f(x)\,dx\] is an absolutely continuous function of $s$, so there exists a function $F$, that we call \textbf{pseudo-rearrangement} of $f$ with respect to $g$, such that 
    \begin{equation*}
       \int_0^sF(t)\,dt=\int_{D(s)}f(x)\,dx.
    \end{equation*}
    
\end{definizione}

The following lemma describes a fundamental property of this rearrangement.
  \begin{lemma}[{\cite[Lemma 2.2]{AT}}]\label{pseudorearr}
        Let $f\in L^p(\Omega),p>1$. There exists a sequence $\{F_k(s)\}$ of functions which have the same rearrangement $f^*$ of $f$, such that \[F_k(s)\rightharpoonup F(s)\quad\text{in }L^p(0,|\Omega|)\]
        If $f\in L^1(\Omega)$, it follows that: \[\lim_k\int_0^{|\Omega|}F_k(s)g(s)\,ds=\int_0^{|\Omega|}F(s)g(s)\,ds\] for every function $g$ belonging to the space $BV(0,|\Omega|)$.
    \end{lemma}

In the special case where $g^*$ is strictly decreasing on $(0,|\Omega|)$, the rearrangement can be constructed as in \cite{CF}, where further details are provided.

\begin{definizione}
    Assume that $g^\ast$ is strictly decreasing in $(0,\abs{\Omega})$. We define $f_g:\Omega \to [0,+\infty)$ as
    $$f_g(x)= f^\ast(\mu_g(g(x))) \qquad\forall x \in \Omega.$$
\end{definizione}
For every $t > 0$, one easily verifies that
\begin{equation}
    \label{equalityh-lprel}
    \Set{x \in \Omega : \, f_g(x) >t}= \Set{x \in \Omega: \abs{g(x)} > g^\ast(\mu_f(t))},
\end{equation}
up to a set of measure zero, and that
$$(f_g)^\ast(s)= f^\ast(s) \qquad a.e. \,  s \in (0, \abs{\Omega}).$$

\begin{oss}
\label{remhardi=}
    Recalling that equality in the Hardy–Littlewood inequality holds when condition \eqref{equalityh-lprel} is satisfied, we obtain:
\begin{equation*}
        \int_\Omega f_g(x)\abs{g(x)}\, dt = \int_0^{\abs{\Omega}} f^\ast(s)g^\ast(s)\, ds=\int_0^{\abs{\Omega}} f_\ast(s)g_\ast(s)\, ds.
    \end{equation*}
\end{oss}

\subsection{Rearrangement of the gradient}

This section recalls classical results on the rearrangement of the gradient, which are instrumental in deriving quantitative versions of \eqref{GN}.
\begin{theorem}
    % \label{Giarrusso_Nunziante}
    Let $\Omega $ be a bounded open set of $\R^n$, $n\geq 2$, and let $\Omega^{\sharp}$ be the centered ball.
    
  \begin{comment}
        Let us suppose that:
    \begin{enumerate}
        \item[a.] $f \colon \Omega \times \R\to \R$ is a measurable function, with $f(x,u) \leq f(x,0) \, \forall (x,u) \in \Omega \times \R$;
        \item[b.]$f_\sharp(x,0)$ is the increasing rearrangement of $f(x,0)$;
        \item[c.]$H \colon \R^n \to \R$ is a measurable, non-negative function;
        \item[d.]$K \colon [0,+\infty) \to [0,+\infty)$ be a strictly increasing real-valued function such that $$K(\abs{y}) \leq H(y) \qquad \forall y \in \R^n ;$$
        \item[f.]$K^{-1}(f(x,0)) \in L^p(\Omega),\quad p\geq1$;
    \end{enumerate}
  \end{comment}
    Let $f \colon \Omega \times \R\to \R$ be a measurable function with $f(x,u) \leq f(x,0) \,\, \forall(x,u) \in \Omega \times \R$, and let $f_\sharp(x,0)$ be the increasing rearrangement of $f(x,0)$. Let $H \colon \R^n \to \R$ be a measurable, non-negative function and let $K \colon [0,+\infty) \to [0,+\infty)$ be a strictly increasing real-valued function such that
    \[
        0 \leq K(\abs{y}) \leq H(y) \qquad \forall y \in \R^n \qquad \text{ and } K^{-1}(f(x,0)) \in L^p(\Omega),\quad p\geq1.
    \]
    Let $v \in W_0^{1,p}(\Omega)$ be a  solution to
    \[
        \begin{cases}
		H(\nabla v) = f(x,u) &\text{a.e. in }\Omega \\
		v = 0 &\text{on } \partial \Omega
	\end{cases}
        ,
    \]
  and let $z \in W_0^{1,p}(\Omega^{\sharp})$ be the unique spherically decreasing solution to
    \[
	\begin{cases}
		K(\abs{\nabla z}) = f_{\sharp}(x,0) & \text{a.e. in } \Omega^{\sharp} \\
		z = 0 & \text{on } \partial \Omega^{\sharp}
	\end{cases}
        ,
    \]
    then, it holds
    \begin{equation*}
	%\label{eq_Giarrusso_Nunziante}
	\norma{v}_{L^1(\Omega)} \leq \norma{z}_{L^1(\Omega^{\sharp})}.
    \end{equation*}
\end{theorem}

Moreover, the following rigidity result is proved in \cite{BM,M} in the case $H(\cdot) = K(\cdot) = |\cdot|$ and $f(x,u) = f(x)$:

\begin{theorem}
    % \label{Mercaldo}
    Let $\Omega \subset \R^n$ be a bounded open set, and let $v \in W_0^{1,1}(\Omega)$ be a non-negative function. Denote by $f(x) = \abs{\nabla v}(x)$ and by $z \in W_0^{1,1}(\Omega^{\sharp})$ the spherically decreasing solution to
    \[
        \abs{\nabla z}(x) = f_{\sharp}(x)\quad  \text{in } \Omega^\sharp.
    \]
    If $\norma{v}_{L^1} = \norma{z}_{L^1}$ then there exists $x_0 \in \R^n$ such that $\Omega = x_0 + \Omega^{\sharp}$, $f=f_{\sharp}(\cdot \, + x_0)$ and $v = z(\cdot \, + x_0)$.
\end{theorem}
\subsection{Some quantitative inequalities}

To prove our main theorem, we first recall several quantitative results that will be used throughout the article.  
The starting point is the quantitative isoperimetric inequality, established in \cite{FMP} (see also \cite{CL,Fuglede,H}).

We recall that the \emph{Fraenkel asymmetry}, denoted by $\alpha(\Omega)$, is a scaling-invariant functional depending only on the shape of $\Omega$. It takes values in $[0,2]$ and will serve as our asymmetry index: it is zero if and only if $\Omega$ is a ball.  Hence, the sharp quantitative isoperimetric inequality can be stated as follows:
\begin{theorem}[Quantitative isoperimetric inequality]
    \label{quant_isop_prop}
    There exists a constant $\gamma_n$ such that,  for any measurable set $\Omega$ of finite measure
    \begin{equation}\label{quant_isop}
        P(\Omega)\geq n\omega_n^{\frac{1}{n}}\abs{\Omega}^{\frac{n-1}{n}}\left(1+\dfrac{\alpha^2(\Omega)}{\gamma_n}\right).
    \end{equation}

\end{theorem}
The dimensional constant $\gamma_n$ is explicitly computed in \cite{FMP2}.  

When analyzing the asymmetry of our domain in Section~\ref{section3}, we will apply the quantitative isoperimetric inequality \eqref{quant_isop} to the superlevel sets of the modulus of a solution to \eqref{original_u}.  
The main difficulty lies in estimating how the asymmetry changes from the whole domain to the superlevel sets, for which we use the following result (see also \cite[Lemma 2.8]{BD}).
\begin{lemma}\label{propasi}
    Let $\Omega\subset\R^n$ be an open set with finite measure and $U\subset\Omega$ a subset of positive measure such that \[\frac{|\Omega\backslash U|}{|\Omega|}\leq\frac{1}{4}\alpha(\Omega)\]
    then holds \[\alpha(U)\geq\frac{1}{2}\alpha(\Omega)\]
\end{lemma}
Another key ingredient is the quantitative version of P\'olya-Szeg\H o
inequality, proved in \cite{CFET}.
\begin{theorem}[Quantitative P\'olya-Szeg\H o
inequality]
% \label{polya_quant}
    Let $u\in W^{1,2}(\R^n)$, $n\geq 2$. Then there exist  positive constants $r$, $s$ and $C$, depending only on $n$,  such that
\begin{equation*}
   \inf_{x_0 \in \R^n} \dfrac{\displaystyle{\int_{\R^n} \abs{u(x)\pm u^\sharp(x+x_0)} \, dx }}{\abs{\{\abs{u}>0\}}^{\frac{1}{n}+\frac{1}{2}}\norma{\nabla u^\sharp}_2}\leq C(n) \left[M_{u}(E(u)^r)+E(u)\right]^s,
\end{equation*}
where 
\begin{equation*}
E(u)= \frac{\displaystyle{\int_{\R^n} \abs{\nabla u}^2}}{\displaystyle{\int_{\R^n} |\nabla u^{\sharp}|^2}}-1 \qquad \text{ and } \qquad M_{u}(\delta)=\dfrac{\abs{\left\{|\nabla u|<\delta\right\}\cap \left\{0<\abs{u}<||u||_\infty\right\}}}{\abs{\{\abs{u}>0\}}}.
\end{equation*}
\end{theorem}

Finally, the main tool to obtain \eqref{GN} is the Hardy-Littlewood inequality \eqref{hardy_littlewood}. To this end, we will use a quantitative version of the Hardy–Littlewood inequality, for which we first introduce some notation, presenting a simplified formulation for clarity although the result was established in a more general setting in \cite{CF}

We define the following Lorentz-type norm:
\begin{equation*}
\normab{g}_{\Lambda^1_1(\Omega)}=\int_0^{\abs{\Omega} }g^\ast(s)\, d\theta(s) + \theta(0^+) \normab{g}_\infty  ,
\end{equation*}
where, for $s\in [0, \abs{\Omega})$,
\begin{equation*}
\theta(s)=\text{ess}\sup_{\sigma \in [0,s)} \frac{1}{-(h^\ast)'(\sigma)}.
\end{equation*}

\begin{theorem}[Quantitative Hardy Littlewood inequality]\label{cianchi_hardy}
    Let  $h$ and $g$ be two non negative measurable functions such that $h g\in L^1(\Omega)$ and let us assume that  $\theta(s)<\infty$ for every $0\leq s< |\Omega|$ and $\norma{g}_{\Lambda ^1_1(\Omega)}<\infty$. Then,
    it holds
    \begin{equation*}
        \int_{\Omega}h(x)g(x) \;dx+\dfrac{1}{2^{2}e} \norma{g}^{-1}_{\Lambda ^1_1(\Omega)}\norma{g-g_h}^{2}_{L^1(\Omega)} \leq \int_0^{|\Omega|} h^\ast(s) g^\ast(s) \;ds.
    \end{equation*}
\end{theorem}

\subsection{Rescaling properties}
\label{riscalamento}

    To simplify the proofs, we make the following two assumptions:
    \begin{align*}
     \abs{\Omega}=1,  \qquad \qquad \norma{f}_\infty=1. 
    \end{align*}

    Even in the general case, we can recover this assumption by setting
    $$a=\dfrac{\abs{\Omega}^{-\frac{1}{n}}}{\norma{f}_\infty}, \quad \quad b= \abs{\Omega}^{-\frac{1}{n}}.$$
    If $u$ is a solution to \eqref{original_u} and $u^\text{G}$ is the unique spherically decreasing solution to \eqref{original_v}, we define the functions $w(x)= au\left(\frac{x}{b}\right)$ and $z(x)= au^G\left(\frac{x}{b}\right)$. Then, setting $g=f \frac a b$, they satisfy 
    
    $$\begin{cases}
        |\nabla w|=g & \text{a.e. in } \Tilde{\Omega}\\
        w=0 & \text{on } \partial\Tilde{\Omega},
    \end{cases} \qquad \qquad \begin{cases}
        |\nabla z|=g_\sharp & \text{a.e. in } \Tilde{\Omega}^\sharp\\
        z=0 & \text{on } \partial\Tilde{\Omega}^\sharp,
    \end{cases}$$
   with $|\Tilde{\Omega}|=1$ and $\norma{g}_\infty=1$.  
Furthermore, the following holds

  \begin{align*}
    \alpha(\tilde \Omega) &= \alpha(\Omega); \\
      \norma{z}_1-\norma{w}_1 &=\dfrac{\abs{\Omega}^{-1-\frac{1}{n}}}{\norma{f}_\infty}\left( ||u^G||_1-||u||_1\right);\\
      \normab{w\pm w^\sharp}_1 &=\dfrac{\abs{\Omega}^{-1-\frac{1}{n}}}{\norma{f}_\infty} \normab{u\pm u^\sharp}_1;\\
      \min g&= \frac{\min f}{\normab{f}_\infty}.
 \end{align*}

Moreover, if $\mu_w(t)$ and $\mu_u(t)$ are the distribution functions of $w$ and $u$ respectively, then
$$
\mu_w(t)= \abs{\Omega}^{-1}\mu_u\left(\frac{t}{a}\right),
$$
which yields
$$
  \normab{g-g_w}_1 =\dfrac{1}{\abs{\Omega}\norma{f}_\infty} \normab{f-f_u}_1.
$$

\section{Asymmetry of the domain}
\label{section3}
Thanks to the rescaling properties in Subsection \ref{riscalamento}, we will make the additional assumptions that $|\Omega| =||f||_\infty = 1$. The general result will be recovered through scaling arguments.

Our first step in proving Theorem \ref{main_theorem} is to study the asymmetry of the set $\Omega$. To this aim we recall the definition of $s_\Omega$, given in \cite{BD},
    \begin{equation}
        \label{somega}
        s_\Omega=\sup\left\{t\geq0:\mu_u(t)\geq1-\frac{\alpha(\Omega)}{4}\right\},
    \end{equation}
    and we prove an intermediate Lemma involving the function $g$, that is the unique spherically decreasing solution to the problem \begin{equation}\label{defg}
        \begin{cases}
            |\nabla g|=F(\omega_n|x|^n)\quad&\text{in }\Omega^\sharp \\ g=0\quad&\text{on }\partial\Omega^\sharp,
        \end{cases}
       \end{equation}
       
  where $F$ is the pseudo-rearrangement of $f$ with respect to $u$ in the sense of definition \ref{pseudo}.
  In \cite{GN}, the authors prove the pointwise inequality $$u^\sharp(x)\leq g(x) \qquad x \in \Omega^\sharp.$$
  The validity of the pointwise inequality is way stronger than the $L^1$ comparison, and subsequently, in the proof of Proposition \ref{main_theorem_part1}, we link $\normab{g}_1$ to  $\normab{u^\text{G}}_1$.
\begin{lemma}
Let $\Omega$ be a bounded open set of $\R^n$, $n\geq 2$, and let $f\in L^\infty(\Omega)$  be a non-negative function. Suppose that $u$ is a solution to \eqref{original_u}. Then, it holds
    \[
    \normab{g}_{L^1(\Omega^\sharp)}- \normab{u}_{L^1(\Omega)}\geq s_\Omega\frac{\alpha(\Omega)^2}{8\gamma_n},\]
    where $\gamma_n$ is the constant appearing in Theorem \ref{quant_isop_prop} and $g$ is the function defined in \eqref{defg}.
\end{lemma}
\begin{proof}
Firstly, we observe that, since $g$ is  spherically decreasing, we have the following explicit expression
\[g(s)=\int_s^{|\Omega|}\frac{F(t)}{n\omega_n^\frac{1}{n} t^{1-\frac{1}{n}}}dt.\]
    By Fleming-Rishel Formula \eqref{flemingrishel} and the quantitative isoperimetric inequality \eqref{quant_isop}, we obtain\begin{equation}\label{perimetri}
-\frac{d}{dt}\int_{\abs{u}>t}|\nabla u|dx=P(\abs{u}>t)\geq n\omega_n^\frac{1}{n}\mu_u(t)^{1-\frac{1}{n}}\left(1+\frac{\alpha({\abs{u}>t})^2}{\gamma_n}\right),   
\end{equation}
and since $u$ is a solution to \eqref{original_u}, we have\begin{equation}\label{eqint}-\frac{d}{dt}\int_{\abs{u}>t}|\nabla u|dx=-\frac{d}{dt}\int_{\abs{u}>t}f(x)dx=-\frac{d}{dt}\int_0^{\mu_u(t)}F(\tau)d\tau=F(\mu_u(t))(-\mu_u'(t)).\end{equation}
By combining \eqref{perimetri} and \eqref{eqint}, it follows \begin{equation*}
    \frac{\alpha({\abs{u}>t})^2}{\gamma_n}\leq\frac{F(\mu_u(t))(-\mu_u'(t))}{n\omega_n^\frac{1}{n}\mu_u(t)^{1-\frac{1}{n}}} -1.
\end{equation*}
Integrating both sides from $0$ to $t$, we have \[\int_0^t \frac{F(\mu_u(\tau))(-\mu_u'(\tau))}{n\omega_n^\frac{1}{n}\mu_u(\tau)^{1-\frac{1}{n}}}d\tau-t\geq\int_0^t \frac{\alpha({\abs{u}>\tau})^2}{\gamma_n}d\tau,\] which gives,  by the definition of rearrangement \ref{decreasing:rear} 
\begin{equation}\label{diffrior}
    g(s) -u^*(s)=\int_s^{|\Omega|}\frac{F(t)}{n\omega_n^\frac{1}{n} t^{1-\frac{1}{n}}}dt-u^*(s)\geq\int_0^{u^\ast(s)}\frac{\alpha({\abs{u}>t})^2}{\gamma_n}dt.
\end{equation}
In order to apply the Lemma \ref{propasi} to the integrand on the right hand side of \eqref{diffrior}, we define the set \[A=\left\{t\geq0:\mu_u(t)\geq1-\frac{\alpha({\Omega})}{4}\right\}.\] 
The set $A$ is not empty if $\alpha(\Omega)>0$, and it is in fact an interval. Moreover, for every $t\in A$ it holds \[\frac{|\Omega\backslash\{\abs{u}>t\}|}{|\Omega|}=|\Omega\backslash\{\abs{u}>t\}|=1-\mu_u(t)\leq\frac{\alpha({\Omega})}{4},\] and by Lemma \ref{propasi}, we have \[\alpha({\abs{u}>t})\geq\frac{\alpha({\Omega})}{2}.\]

Thus, since $s_\Omega = \sup A$, for $t<s_\Omega$, we have 
$$\mu_u(t) \geq 1-\frac{\alpha({\Omega})}{4} \implies \alpha(\abs{u}>t) < \frac{\alpha(\Omega)}{2}, $$
hence the integral on the right-hand side of \eqref{diffrior} can be estimated from below as
\begin{equation}\label{intasi}
    \int_0^{u^\ast(s)}\frac{\alpha({\abs{u}>t})^2}{\gamma_n}dt\geq \int_0^{\min\{s_\Omega,u^*(s)\}}\frac{\alpha({\abs{u}>t})^2}{\gamma_n}dt\geq\min\{s_\Omega,u^*(s)\}\frac{\alpha({\Omega})^2}{4\gamma_n}.
\end{equation}

Let us combine \eqref{diffrior} and \eqref{intasi} and let us integrate in $[0,1]= [0, \abs{\Omega}]$. Hence we get
$$
\|g\|_1-\|u\|_1\geq\frac{\alpha({\Omega})^2}{4\gamma_n}\int_0^1\min\{s_\Omega, u^\ast(s)\} \, ds.
$$
Last, the integral of the minimum can be estimated as follows
$$
\int_0^1\min\{s_\Omega, u^\ast(s)\} \, ds \geq\int_0^{1-\frac{\alpha(\Omega)}{4}}\min\{s_\Omega, u^\ast(s)\} \, ds = s_\Omega \left(1-\frac{\alpha(\Omega)}{4}\right), $$
as  for $s \leq 1-\frac{\alpha(\Omega)}{4}$ it holds $u^\ast(s) \geq u^\ast\left(1-\frac{\alpha(\Omega)}{4}\right)= s_\Omega$, obtaining
\[\|g\|_1-\|u\|_1\geq s_\Omega\frac{\alpha({\Omega})^2}{8\gamma_n}.\]
\end{proof}

We are now in position to prove that the deficit $\displaystyle{\normab{u^\text{G}}_{L^1(\Omega^\sharp)}-\normab{u}_{L^1(\Omega)}}$ controls the asymmetry index of $\Omega$.

\begin{prop}
\label{main_theorem_part1}
Let $\Omega$ be a bounded open set of $\R^n$, $n\geq 2$,  and let $f\in L^\infty(\Omega)$ with $f\geq m >0$ be a non-negative function.
Suppose that $u$ and $u^\text{G}$ are solutions to \eqref{original_u} and \eqref{original_v}, respectively. 

    Then there exists an explicit constant $C_1=C_1(n,m)$ such that the following inequality holds
    \[C_1\alpha(\Omega)^3 \leq \normab{u^\text{G}}_{L^1(\Omega^\sharp)}-\normab{u}_{L^1(\Omega)}.\]
\end{prop}
\begin{proof}
    Let us recall that if $F$ is pseudo-rearrangement of $f$ with respect to $u$, by the mean value theorem and the hypothesis $f\geq m >0$, it holds for $a.e.\,  s \in [0, 1]$ \begin{equation}
    \label{fmagm}
        \begin{aligned}
            F(s)&=\lim_{h \to 0}\frac{1}{h}\int_s^{s+h}F(t)dt\\&=\lim_{h \to 0}\frac{1}{h}\int_{D(s+h)\backslash D(s)}f(x)\,dx\\&\geq \lim_{h \to 0}\frac{1}{h}\int_{D(s+h)\backslash D(s)}m \,dx \\&=\lim_{h \to 0}\frac{s+h-s}{h}m=m.
        \end{aligned}
    \end{equation} 

    Let $g$ be the unique spherically decreasing solution to \eqref{defg} 
       and let $\mu_g$  denote its distribution function.
       
By \eqref{fmagm}, and in view of \eqref{defg}
        the gradient $\abs{\nabla g}$
       does not vanish almost everywhere in $\Omega^\sharp$.
       Moreover, since 
$g$ is spherically decreasing, its superlevel sets are concentric balls, for which the isoperimetric inequality holds as  an equality.
    
    Hence, for $\mu_g$ we have \[1=\frac{F(\mu_g(t))(-\mu_g'(t))}{n\omega_n^\frac{1}{n}\mu_g(t)^{1-\frac{1}{n}}}\Rightarrow -\mu_g'(t)=\frac{n\omega_n^\frac{1}{n}\mu_g(t)^{1-\frac{1}{n}}}{F(\mu_g(t))}.\] 
    Following the approach in \cite{Kim},  the absolutely continuity of $\mu_g$ ensures that there exists $t_1$  such that \[\mu_g(2t_1)=1-\frac{\alpha(\Omega)}{8}.\] 
   Hence we have
   \begin{align*}
\frac{\alpha(\Omega)}{8}&=1-\left(1-\frac{\alpha(\Omega)}{8}\right)=\mu_g(0)-\mu_g(2t_1)=\\&=-\int_0^{2t_1}\mu_g'(s)ds=\int_0^{2t_1}\frac{n\omega_n^\frac{1}{n}\mu_g(s)^{1-\frac{1}{n}}}{F(\mu_g(s))}ds\leq\frac{n\omega_n^\frac{1}{n}2t_1}{m}.
\end{align*}
Finally, we obtain  that
\begin{equation}\label{stimat1}
    t_1\geq m\frac{\alpha(\Omega)}{16n\omega_n^\frac{1}{n}}.
\end{equation}
Now let us distinguish two cases:
\begin{itemize}
    \item if $s_\Omega\geq t_1$, then it directly follows from \eqref{stimat1} that \[\|g\|_1-\|u\|_1\geq s_\Omega\frac{\alpha({\Omega})^2}{8\gamma_n}\geq t_1\frac{\alpha({\Omega})^2}{8\gamma_n} \geq  \frac{m}{128n\gamma_n\omega_n^\frac{1}{n}}\alpha({\Omega})^3. \]
    \item if $s_\Omega<t_1$, then, since the pointwise inequality $u^\sharp(x)\leq g(x)$ holds in $\Omega^\sharp$, we have
    \begin{equation} \label{diffdistr}
        \|g\|_1-\|u\|_1=\int_0^\infty \left(\mu_g(t)-\mu_u(t)\right)dt\geq\int_{t_1}^{2t_1}\left(\mu_g(t)-\mu_u(t)\right)dt.
    \end{equation}
     Moreover, by $s_\Omega<t_1\leq t\leq 2t_1$ and by the definition of both $s_\Omega$ \eqref{somega} and $t_1$, it follows that \begin{equation}
         \label{stimadistr} \mu_g(t)\geq\mu_g(2t_1)=1-\frac{\alpha(\Omega)}{8} ,\qquad \mu_u(t)\leq 1-\frac{\alpha(\Omega)}{4}.
     \end{equation} Therefore, by combining \eqref{diffdistr} and \eqref{stimadistr} \begin{align*}
         \|g\|_1-\|u\|_1\geq\int_{t_1}^{2t_1}\left(\mu_g(t)-\mu_u(t)\right)dt\geq \\ \geq \frac{\alpha(\Omega)}{8}\int_{t_1}^{2t_1}dt\geq  \frac{t_1}{8}\alpha(\Omega)^2\geq  \frac{m}{256n\omega_n^\frac{1}{n}}\alpha(\Omega)^3.
     \end{align*}
\end{itemize}
To conclude we show that $$\normab{u^\text{G}}_1-\normab{u}_1 \geq \normab{g}_1-\normab{u}_1.$$
Indeed, Lemma \ref{pseudorearr} and Hardy-Littlewood inequality brings to
\begin{align*}
    \normab{g}_{L^1(\Omega^\sharp)}=\frac{1}{n\omega_n^{\frac{1}{n}}}\int_0^{1}\int_s^{1}\frac{F(t)}{t^{1-\frac{1}{n}}}dtds=\frac{1}{n\omega_n^{\frac{1}{n}}}\int_0^{1}F(t)t^\frac{1}{n}dt&=\frac{1}{n\omega_n^{\frac{1}{n}}}\lim_k\int_0^{1}f_k(t)t^\frac{1}{n}dt\\&\leq \frac{1}{n\omega_n^{\frac{1}{n}}}\int_0^{1}f_*(t)t^\frac{1}{n}dt=\normab{u^{\text{G}}}_{L^1(\Omega^\sharp)},
\end{align*}
and the thesis follows.
\end{proof}

\section{Almost radiality of the solutions}
\label{section4}
Considering  Subsection \ref{riscalamento}, we will assume, without loss of generality, that $|\Omega| =||f||_\infty = 1$. Moreover, let us define $\varepsilon:= \normab{u^G}_1-\normab{u}_1$.

Our first step in proving the almost radiality of solutions is to prove Theorem \ref{quantitbella}. 
\begin{proof}[Proof of Theorem \ref{quantitbella}]  
 Let us observe that as the function $u^\sharp$ is radial, so the function $\abs{\nabla u^\sharp}$ is. This implies that one can construct the family of set $D(s)$, defined in Section \ref{section2}, such that $\forall s \in [0,1]$, $D(s)$ is a ball or an annulus. 

Moreover, let us remark that it is always possible to find a function $h$ such that 
\begin{enumerate}
    \item $h_\ast(t)=t^{\frac{1}{n}};$
    \item the family of $D(s)$ are the superlevel sets of $h$;
    \item for every $t, \tau > 0$, either \begin{equation}
    \label{equalityh-l}
       ( \{|\nabla u^\sharp| >t\}\subset\{\abs{h}>\tau\})\text{ or }(\{|\nabla u^\sharp| >t\}\supset\{\abs{h}>\tau\}).
    \end{equation}
\end{enumerate}

As in remark \ref{remhardi=}, we know that equality in Hardy-Littlewood inequality occurs when condition \eqref{equalityh-l} is satisfied, hence
\begin{equation*}
        \int_0^{1} |\nabla u^\sharp|_*t^{\frac{1}{n}} \, dt = \int_{\Omega^{\sharp}}|\nabla u^\sharp|h(x) \, dx.
    \end{equation*}
    
By coarea formula \eqref{coarea}, isoperimetric inequality and the properties of rearrangements, we have
\begin{equation*}
\begin{aligned}
    \int_0^{1} |\nabla u^\sharp|_*t^{\frac{1}{n}} \, dt = \int_{\Omega^{\sharp}}|\nabla u^\sharp|h(x) \, dx = \int_0^{\infty} \int_{u^*=s} h(t(s))\,d\mathcal{H}^{n-1} \,ds\\=\int_0^{\infty} \mathcal{H}^{n-1}(\{u^*=s\}) h(t(s)) \,ds
    \leq \int_0^{\infty} \mathcal{H}^{n-1}(\{u=s\}) h(t(s)) \,ds\\
   = \int_0^{\infty} \int_{u=s} h(t(s))\,d\mathcal{H}^{n-1} \,ds
   = \int_{\Omega}|\nabla u|h_u(x) \, dx
   \leq \int_0^{1} |\nabla u|_*t^{\frac{1}{n}} \, dt,
\end{aligned}
\end{equation*}
where $h_u$ is a function with the same superlevel sets as $u$ and whose increasing rearrangement is $t^{\frac{1}{n}}$.

    Hence by definition
    \begin{equation}
    \label{usharpg}
         \begin{gathered}
    \|(u^\sharp)^G\|_1=\frac{1}{n\omega_n^{\frac{1}{n}}}\int_0^{1}\int_s^{1}|\nabla u^\sharp|_*(t)t^{\frac{1}{n}-1}dtds=\\=\frac{1}{n\omega_n^{\frac{1}{n}}}\int_0^{1}|\nabla u^\sharp|_*(t)t^{\frac{1}{n}}dt\leq \frac{1}{n\omega_n^{\frac{1}{n}}}\int_0^{1}|\nabla u|_*(t)t^{\frac{1}{n}}dt=\|u^G\|_1
    \end{gathered}
    \end{equation}
    While the first inequality in \eqref{sharpG} is the inequality \eqref{GN} applied to $u^\sharp$.
\end{proof}

As a direct consequence of \eqref{usharpg}, we obtain the following corollary.
\begin{corollario}
 Let $\Omega$ be a bounded open set of $\R^n$, $n\geq 2$,  and let $u$ be a solution to \eqref{original_u}. Then 
 \[\int_0^{1}\left[|\nabla u|_*-|\nabla u^\sharp|_*\right]t^\frac{1}{n}dt<{n\omega_n^{\frac{1}{n}}} \varepsilon .\]
\end{corollario}

The next two lemmas show that the Pólya–Szegő inequality is almost satisfied, by means of the set 
$I$, defined as follows
\begin{equation}
    \label{defI}
    I=\left\{t\in[0,1]:\left[|\nabla u|_*(t)-|\nabla u^\sharp|_*(t)\right]t^\frac{1}{n}>\varepsilon^\alpha\right\}.
\end{equation}

\begin{lemma}
\label{lemma4.3}   Let $\Omega$ be a bounded open set of $\R^n$, $n\geq 2$,  and let $u$ be a solution to \eqref{original_u}. Let $I$ be the set defined in \eqref{defI}, then for all $0<\alpha<1$ it holds $|I|<n \omega_n^\frac{1}{n}\varepsilon^{1-\alpha}$.
%For all $\alpha>0$, if $I$ is the set defined in \eqref{defI}, then $|I|<\varepsilon^{1-\alpha}$.
\end{lemma}
\begin{proof} By the previous corollary
    \begin{gather*}
        |I|=\int_Idt=\frac{1}{\varepsilon^\alpha}\int_I\varepsilon^\alpha dt<\frac{1}{\varepsilon^\alpha}\int_I \left[|\nabla u|_*(t)-|\nabla u^\sharp|_*(t)\right]t^\frac{1}{n}dt<\\<\frac{1}{\varepsilon^\alpha}\int_0^{1} \left[|\nabla u|_*(t)-|\nabla u^\sharp|_*(t)\right]t^\frac{1}{n}dt<n \omega_n^\frac{1}{n}\varepsilon^{1-\alpha}
    \end{gather*}
\end{proof}

\begin{lemma}
\label{lemmaEu}
Let $\Omega$ be a bounded open set of $\R^n$, $n\geq 2$,  and let $u$ be a solution to \eqref{original_u}. 
   Then,  if $\displaystyle \varepsilon \leq \frac{m^4}{64n^2 \omega_n^\frac{2}{n}}$, we have
   \[E(u)=\frac{\displaystyle{\int_\Omega|\nabla u|^2dx}}{\displaystyle{\int_{\Omega^\sharp}|\nabla u^\sharp|^2dx}}-1\leq \frac{8n\omega_n^\frac{1}{n}}{m^2} \varepsilon^{\frac 1 2} .\]
\end{lemma}
\begin{proof}

By the properties of the rearrangements and the P\'olya-Szeg\H o  inequality in the case $p=\infty$, Lemma \ref{asscontmu}, we get
$$
\lVert\nabla u^\sharp\rVert_\infty \leq\norma{\nabla u}_\infty  = \norma{f}_\infty =1,
$$
hence 
    \begin{align*}
        \int_\Omega|\nabla u|^2dx-\int_{\Omega^\sharp}|\nabla u^\sharp|^2dx&=\int_0^{1}|\nabla u|_*^2-|\nabla u^\sharp|_*^2dx\\&=\int_0^{1}\left(|\nabla u|_*+|\nabla u^\sharp|_*\right)\left(|\nabla u|_*-|\nabla u^\sharp|_*\right)dx \\&\leq 2\int_0^{1}|\nabla u|_*-|\nabla u^\sharp|_*dx.
    \end{align*}
Lemma \ref{lemma4.3} allows us to estimate the last term, indeed
\begin{align*}
    \int_0^{1} |\nabla u|_*-|\nabla u^\sharp|_*&=\int_{[0,1]\cap I} \hspace{-5mm}|\nabla u|_*-|\nabla u^\sharp|_*+\int_{[0,1]\backslash I} \hspace{-5mm}(|\nabla u|_*-|\nabla u^\sharp|_*)t^\frac{1}{n} t^{-\frac 1 n}\\&\leq 2|I|+\varepsilon^\alpha \frac{n}{n-1}\\
    &\leq 2n \omega_n^\frac{1}{n}\varepsilon^{1-\alpha}+\varepsilon^\alpha \frac{n}{n-1}.
\end{align*}
    By choosing $\alpha=\frac{1}{2}$ we get
    \[
    \int_\Omega|\nabla u|^2dx-\int_{\Omega^\sharp}|\nabla u^\sharp|^2dx \leq 4n \omega_n^\frac{1}{n} \varepsilon^{\frac{1}{2}}.
    \]
    Moreover, if 
\[
 \varepsilon \leq \frac{m^4}{64n^2 \omega_n^\frac{2}{n}} \implies 4n \omega_n^\frac{1}{n}\varepsilon^{\frac{1}{2}}\leq \frac{m^2}{2}
\]
we have 
\[
\int_{\Omega^\sharp}|\nabla u^\sharp|^2dx \geq \int_\Omega|\nabla u|^2dx -4 n \omega_n^\frac{1}{n}\varepsilon^{\frac{1}{2}} \geq \frac{m^2}{2}.
\]
The last gives us
\[\frac{\displaystyle{\int_\Omega|\nabla u|^2dx}}{\displaystyle{\int_{\Omega^\sharp}|\nabla u^\sharp|^2dx}}-1\leq   \frac{4n \omega_n^\frac{1}{n} \varepsilon^{\frac 1 2}}{\displaystyle{\int_{\Omega^\sharp}|\nabla u^\sharp|^2dx}} \leq \frac{8n \omega_n^\frac{1}{n}}{m^2} \varepsilon^{\frac{1}{2}}. \]

\end{proof}

We are now ready to establish the stability result for 
$u$.

\begin{prop}

\label{main_theorem_part2}
Let $\Omega$ be  a bounded open set of $\R^n$, $n\geq 2$, and let $f\in L^\infty(\Omega)$  be a non-negative function with $$0<m\leq f\leq M.$$  
Suppose that $u$ and $u^\text{G}$ are solutions to \eqref{original_u} and \eqref{original_v}, respectively. Then there exist positive constants $ \theta=\theta(n)$ and    $ C_2:=  C_2(n, m)$ such that
\begin{equation*}
     C_2 \inf_{x_0\in\R^n}\normab{u\pm u^\sharp(\cdot+x_0)}_{L^1(\R^n)}^{\theta } \leq \normab{u^G}_1-\normab{u}_1 .
\end{equation*}
\end{prop}
\begin{proof}
We set $\varepsilon:= \normab{u^G}_1-\normab{u}_1$.

Taking into account the quantitative P\'olya-Szeg\H o principle in \cite{CFET}, we immediately get 
\begin{equation}
    \label{immcons} C\inf_{x_0\in\R^n}\normab{u\pm u^\sharp(\cdot+x_0)}_{L^1(\R^n)}\leq \left[M_u(E(u)^r)+ E(u)\right]^s,
\end{equation}

 for some positive constants $C,r$ and $s$.

    In order to conclude, it is necessary to estimate the term 
    \[M_u(E(u)^r)=\frac{|\{|\nabla u|\leq E(u)^r\}\cap\{0<\abs{u}<\normab{u}_\infty\}|}{|\{|u|>0\}|}.\]

To this aim, we will use Lemma \ref{lemmaEu}. Indeed,
    since $1\geq |\nabla u|=f\geq m$ a.e., if $$\varepsilon <\min\left\{ \frac{m^{4+\frac{2}{r}}}{64n^2 \omega_n^\frac{2}{n}} ,\frac{m^4}{64n^2 \omega_n^\frac{2}{n}}\right\}= \frac{m^{4+\frac{2}{r}}}{64n^2 \omega_n^\frac{2}{n}}\implies m>\left(\frac{8n\omega_n^\frac{1}{n}}{m^2}\varepsilon^\frac{1}{2}\right)^r,$$
    we have 
    $$|\nabla u|\geq m>\left(\frac{8n\omega_n^\frac{1}{n}}{m^2}\varepsilon^\frac{1}{2}\right)^r \geq E(u)^r \,\,a.e. \text{ in } \Omega.$$

    Hence $M(E(u)^r)=0$,  and \eqref{immcons} reads 
\begin{equation*}
     C\inf_{x_0\in\R^n}\normab{u\pm u^\sharp(\cdot+x_0)}_{L^1(\R^n)}\leq E^s(u) \leq \frac{(8n\omega_n^\frac{1}{n})^s}{m^{2s}}\varepsilon^{\frac{s}{2}}.
\end{equation*}

Dividing by ${(8n\omega_n^\frac{1}{n})^s}/{m^{2s}}$ and raising everything to the power $2/s$, we have the thesis.

    On the other hand, if $\varepsilon>\frac{m^{4+\frac{2}{r}}}{64n^2 \omega_n^\frac{2}{n}}$ we have 
    \begin{align*}
        \normab{u\pm u^\sharp}_{L^1(\R^n)}\leq2\normab{u}_{L^1(\Omega)}\leq2\normab{u^G}_{L^1(\Omega)}=2\int_0^1f_\sharp(t)t^\frac{1}{n}dt\\\leq 2\frac{n}{n+1}\leq \frac{128n^2 \omega_n^\frac{2}{n}}{m^{4+\frac{2}{r}}}\varepsilon.
    \end{align*}
\end{proof}

Hence, Theorem \ref{main_theorem} can be recovered by putting together Propositions \ref{main_theorem_part1} and \ref{main_theorem_part2}.
\begin{proof}[Proof of Theorem \ref{main_theorem}]

Let us consider $\abs{\Omega}=||f||_\infty=1$, taking into account Propositions \ref{main_theorem_part1} and  \ref{main_theorem_part2}, there exists positive constants $ \theta=\theta(n)$, $C_1=C_1(n,m)$ and  $ C_2:=  C_2(n,  m)$, such that
\begin{equation*}
     C_1\alpha(\Omega)^3 +C_2 \inf_{x_0\in\R^n}\normab{u\pm u^\sharp(\cdot+x_0)}_{L^1(\R^n)}^{\theta } \leq \normab{u^G}_1-\normab{u}_1. 
\end{equation*}

    Defining the functions $w(x)= au\left(\frac{x}{b}\right)$ and $z(x)= au^G\left(\frac{x}{b}\right)$  with $$a=\dfrac{\abs{\Omega}^{-\frac{1}{n}}}{\norma{f}_\infty}, \quad \quad b= \abs{\Omega}^{-\frac{1}{n}},$$
    considering the rescaling properties in Section \ref{riscalamento}, we have
    \begin{equation*}
    \begin{multlined}
         C_1\left(n, \frac{m}{\normab{f}}_\infty\right)\alpha(\Omega)^3 +C_2 \left(n, \frac{m}{\normab{f}}_\infty\right)\left(\frac{\abs{\Omega}^{-1-\frac{1}{n}}}{\normab{f}_\infty}\right)^\theta\inf_{x_0\in\R^n}\normab{u\pm u^\sharp(\cdot+x_0)}_{L^1(\R^n)}^{\theta }\leq  \\ \left(\frac{\abs{\Omega}^{-1-\frac{1}{n}}}{\normab{f}_\infty}\right)\left(\normab{u^G}_1-\normab{u}_1\right). 
    \end{multlined}
\end{equation*}
Dividing by $\frac{\abs{\Omega}^{-1-\frac{1}{n}}}{\normab{f}_\infty}$, we get
the thesis. 

\end{proof}

\section{Property of the source term}
\label{section5}
The last step is to obtain some information on the source term.

\begin{proof}[Proof of Proposition \ref{main_prop}] Let us consider $\abs{\Omega}=||f||_\infty=1$, and let us set $\varepsilon:= \normab{u^G}_1-\normab{u}_1$.

In proof of Theorem \ref{quantitbella}, we show that  
\begin{equation*}
    \int_0^{|\Omega|} f^*(t)(1-t)^{\frac{1}{n}} \, dt-  \int_{\Omega}f h_u(x) \, dx= \int_0^{\abs{\Omega}} |\nabla u|_*t^{\frac{1}{n}} \, dt -\int_{\Omega}|\nabla u|h_u(x) \, dx
   \leq n \omega_n^\frac{1}{n} \varepsilon.
\end{equation*}

Applying the quantitative Hardy-Littlewood inequality \ref{cianchi_hardy}, we have
\begin{equation*}
\norma{f-f_u}_{L^1(\R^n)}\leq \left[4\,e\,\norma{f}_{\Lambda_1^1(\Omega)}\left( \int_0^{|\Omega|} f^*(t)(1-t)^{\frac{1}{n}} \, dt-  \int_{\Omega}f h_u(x) \, dx\right)\right]^{\frac{1}{2}}.
\end{equation*}
By direct calculations, we have that 
%$$
%\norma{f}_{\Lambda_1^1(\Omega)}\leq \frac{\normab{f}_\infty}{n} =\frac{1}{n}.
%$$
$$
\norma{f}_{\Lambda_1^1(\Omega)}=n \normab{f}_\infty=n.$$

Hence, we have
\begin{equation*}
\frac{1}{4en^2 \omega_n^\frac{1}{n}}\norma{f-f_u}_1^{2}\leq \varepsilon.
\end{equation*}
Considering the rescaling properties in Section \ref{riscalamento}, we get the thesis.
\end{proof}

\subsubsection*{Acknowledgements}
We would like to thank Cristina Trombetti and Alba Lia Masiello for their valuable advices that helped us to achieve these results.
\subsubsection*{Declarations}

\paragraph{Funding}
 The authors were partially supported by Gruppo Nazionale per l’Analisi Matematica, la Probabilità e le loro Applicazioni
(GNAMPA) of Istituto Nazionale di Alta Matematica (INdAM). 

\paragraph{Data Availability} All data generated or analysed during this study are included in this published article.
	
\paragraph{Competing Interests} We declare that we have no financial and personal relationships with other people or organizations.

\nocite{Kim}
\bibliographystyle{plain}
\bibliography{bibliografia.bib}

\begin{thebibliography}{10}

\bibitem{ALT}
A.~Alvino, P.-L. Lions, and G.~Trombetti.
\newblock On optimization problems with prescribed rearrangements.
\newblock {\em Nonlinear Anal.}, 13(2):185--220, 1989.

\bibitem{AT}
A.~Alvino and G.~Trombetti.
\newblock Sulle migliori costanti di maggiorazione per una classe di equazioni ellittiche degeneri.
\newblock {\em Ricerche di Matematica}, 27(2):413--428, 1978.

\bibitem{ABMP}
V.~Amato, R.~Barbato, A.L. Masiello, and G.~Paoli.
\newblock The {T}alenti comparison result in a quantitative form.
\newblock {\em Annali Scuola Normale Superiore - Classe di Scienze}, page~31, dec 2024.

\bibitem{AG}
V.~Amato and A.~Gentile.
\newblock On the symmetric rearrangement of the gradient of a {S}obolev function.
\newblock {\em Atti Accad. Naz. Lincei Rend. Lincei Mat. Appl.}, 34(2):433--450, 2023.

\bibitem{AGNT}
V.~Amato, A.~Gentile, C.~Nitsch, and C.~Trombetti.
\newblock On the gradient rearrangement of functions.
\newblock {\em Mathematische Annalen}, April 2024.

\bibitem{AFP}
L.~Ambrosio, N.~Fusco, and D.~Pallara.
\newblock {\em Functions of Bounded Variation and Free Discontinuity Problems}.
\newblock Oxford University PressOxford, March 2000.

\bibitem{BM}
M.~F. Betta and A.~Mercaldo.
\newblock Uniqueness results for optimization problems with prescribed rearrangement.
\newblock {\em Potential Analysis}, 5:183--205, 1996.

\bibitem{BD}
L.~Brasco and G.~De~Philippis.
\newblock Spectral inequalities in quantitative form.
\newblock In {\em Shape {O}ptimization {A}nd {S}pectral {T}heory}, pages 201--281. De Gruyter Open, Warsaw, 2017.

\bibitem{BZ}
J.~E. Brothers and W.~P. Ziemer.
\newblock Minimal rearrangements of {S}obolev functions.
\newblock {\em Journal für die reine und angewandte Mathematik}, 384:153--179, 1988.

\bibitem{C}
A.~Cianchi.
\newblock On the {$L^q$} norm of functions having equidistributed gradients.
\newblock {\em Nonlinear Anal.}, 26(12):2007--2021, 1996.

\bibitem{CFET}
A.~Cianchi, L.~Esposito, N.~Fusco, and C.~Trombetti.
\newblock A quantitative {P}\'{o}lya-{S}zeg\"{o} principle.
\newblock {\em J. Reine Angew. Math.}, 614:153--189, 2008.

\bibitem{CF}
A.~Cianchi and A.~Ferone.
\newblock A strengthened version of the {H}ardy-{L}ittlewood inequality.
\newblock {\em J. Lond. Math. Soc. (2)}, 77(3):581--592, 2008.

\bibitem{CF2}
A.~Cianchi and N.~Fusco.
\newblock Functions of bounded variation and rearrangements.
\newblock {\em Archive for rational mechanics and analysis}, 165:1--40, 2002.

\bibitem{CL}
M.~Cicalese and G.~P. Leonardi.
\newblock A selection principle for the sharp quantitative isoperimetric inequality.
\newblock {\em Arch. Ration. Mech. Anal.}, 206(2):617--643, 2012.

\bibitem{FP}
V.~Ferone and M.~R. Posteraro.
\newblock Maximization on classes of functions with fixed rearrangement.
\newblock {\em Differential Integral Equations}, 4(4):707--718, 1991.

\bibitem{FPV}
V.~Ferone, M.~R. Posteraro, and R.~Volpicelli.
\newblock An inequality concerning rearrangements of functions and {H}amilton-{J}acobi equations.
\newblock {\em Arch. Rational Mech. Anal.}, 125(3):257--269, 1993.

\bibitem{FMP2}
A.~Figalli, F.~Maggi, and A.~Pratelli.
\newblock A mass transportation approach to quantitative isoperimetric inequalities.
\newblock {\em Invent. Math.}, 182(1):167--211, 2010.

\bibitem{Fleming_Rishel}
W.H. Fleming and R.~Rishel.
\newblock An integral formula for total gradient variation.
\newblock {\em Arch. Math. (Basel)}, 11:218--222, 1960.

\bibitem{Fuglede}
B.~Fuglede.
\newblock Stability in the isoperimetric problem.
\newblock {\em Bull. London Math. Soc.}, 18(6):599--605, 1986.

\bibitem{F}
N.~Fusco.
\newblock The quantitative isoperimetric inequality and related topics.
\newblock {\em Bull. Math. Sci.}, 5(3):517--607, 2015.

\bibitem{FMP}
N.~Fusco, F.~Maggi, and A.~Pratelli.
\newblock The sharp quantitative isoperimetric inequality.
\newblock {\em Ann. of Math. (2)}, 168(3):941--980, 2008.

\bibitem{GN}
E.~Giarrusso and D.~Nunziante.
\newblock Symmetrization in a class of first-order {H}amilton-{J}acobi equations.
\newblock {\em Nonlinear Analysis: Theory, Methods \& Applications}, 8(4):289--299, 1984.

\bibitem{GN2}
E.~Giarrusso and D.~Nunziante.
\newblock Comparison theorems for a class of first order {H}amilton-{J}acobi equations.
\newblock {\em Annales de la Facult\'e des sciences de Toulouse : Math\'ematiques}, Ser. 5, 7(1):57--73, 1985.

\bibitem{H}
R.~R. Hall.
\newblock A quantitative isoperimetric inequality in {$n$}-dimensional space.
\newblock {\em J. Reine Angew. Math.}, 428:161--176, 1992.

\bibitem{HN}
W.~Hansen and N.~Nadirashvili.
\newblock Isoperimetric inequalities in potential theory.
\newblock In {\em Proceedings from the {I}nternational {C}onference on {P}otential {T}heory ({A}mersfoort, 1991)}, volume~3, pages 1--14, 1994.

\bibitem{HLP}
G.~H. Hardy, J.~E. Littlewood, and G.~P\'olya.
\newblock {\em Inequalities}.
\newblock Cambridge Mathematical Library. Cambridge University Press, Cambridge, 1988.
\newblock Reprint of the 1952 edition.

\bibitem{K}
S.~Kesavan.
\newblock {\em Symmetrization and applications}, volume~3.
\newblock world scientific, 2006.

\bibitem{Kim}
D.~Kim.
\newblock Quantitative inequalities for the expected lifetime of {B}rownian motion.
\newblock {\em Michigan Math. J.}, 70(3):615--634, 2021.

\bibitem{maggi2012sets}
F.~Maggi.
\newblock {\em Sets of finite perimeter and geometric variational problems}, volume 135 of {\em Cambridge Studies in Advanced Mathematics}.
\newblock Cambridge University Press, Cambridge, 2012.
\newblock An introduction to geometric measure theory.

\bibitem{M}
A.~Mercaldo.
\newblock A remark on comparison results for first order {H}amilton-{J}acobi equations.
\newblock {\em Nonlinear Analysis: Theory, Methods and Applications}, 28(9):1465--1477, 1997.

\bibitem{PS}
G.~P\'olya and G.~Szeg\"o.
\newblock {\em Isoperimetric {I}nequalities in {M}athematical {P}hysics}, volume No. 27 of {\em Annals of Mathematics Studies}.
\newblock Princeton University Press, Princeton, NJ, 1951.

\bibitem{Ta6}
G.~Talenti.
\newblock {\em On functions whose gradients have a prescribed rearrangement}, pages 559--571.
\newblock World Scientific, 1994.

\bibitem{tran}
H.~V. Tran.
\newblock {\em Hamilton-{J}acobi equations---theory and applications}, volume 213 of {\em Graduate Studies in Mathematics}.
\newblock American Mathematical Society, Providence, RI, [2021] \copyright 2021.

\end{thebibliography}

\Addresses

\end{document}